\documentclass{article}

\usepackage{PRIMEarxiv}

\usepackage[utf8]{inputenc} % allow utf-8 input
\usepackage[T1]{fontenc}    % use 8-bit T1 fonts
\usepackage{hyperref}       % hyperlinks
\usepackage{url}            % simple URL typesetting
\usepackage{booktabs}       % professional-quality tables
\usepackage{amsfonts}       % blackboard math symbols
\usepackage{nicefrac}       % compact symbols for 1/2, etc.
\usepackage{microtype}      % microtypography
\usepackage{lipsum}
\usepackage{fancyhdr}       % header
\usepackage{graphicx}       % graphics
\graphicspath{{media/}}     % organize your images and other figures under media/ folder
\usepackage{amsmath}
\usepackage{algorithm}
\usepackage{algpseudocode}
\usepackage{amsthm}

\newtheorem{theorem}{Theorem}[section]
\newtheorem{proposition}[theorem]{Proposition}

\newtheorem{remark}[theorem]{Remark}

\newcommand{\ett}{{\bf 1}} 
\newcommand{\mR}{{\mathbb R}}

\DeclareMathOperator*{\argmin}{argmin} % no space, limits underneath in displays

\DeclareMathOperator*{\ve}{vec}

\title{A cutting plane algorithm for globally solving\\ low dimensional k-means clustering problems
}

\author{
  Martin Ryner, Jan Kronqvist, Johan Karlsson \\
  Division of Numerical Analysis, Optimization and Systems Theory,\\ Department of Mathematics, \\KTH Royal Institute of Technology, Stockholm, Sweden\\
  \texttt{martinrr@kth.se, jankr@kth.se, johan.karlsson@math.kth.se} \\
}

\begin{document}
\maketitle

\begin{abstract}
Clustering is one of the most fundamental tools in data science and machine learning, and k-means clustering is one of the most common such methods. There is a variety of approximate algorithms for the k-means problem, but computing the globally optimal solution is in general NP-hard. 
In this paper we consider the k-means problem for instances with low dimensional data 
and formulate it as a structured concave assignment problem. This allows us to exploit the low dimensional structure and solve the problem to global optimality within reasonable time for large data sets with several clusters. 
The method builds on iteratively solving a small concave problem and a large linear programming problem. This gives a sequence of feasible solutions along with bounds which we show converges to zero optimality gap.
The paper combines methods from global optimization theory to accelerate the procedure, and we provide numerical results on their performance.
\end{abstract}

% keywords can be removed

\section{Introduction}

\subsection{Background}

Clustering is a fundamental problem in the field of machine learning and data analysis for partitioning a dataset in order to discover underlying patterns or structures. One of the most widely used models is the k-means problem \cite{macqueen1967some,lloyd1982least}, which builds on describing the data as well as possible using centroids. More specifically, for a data set $\{x_i\}_{i = 1}^n \subset \mathbb{R}^d$, and a specified number of clusters $k$, the goal is to find centroids $\{y_j\}_{j=1}^k \subset \mathbb{R}^d$ that minimize the sum of squared distances between the data points and their respective centroids, i.e.,
\begin{align}
\min_{\{y_j\}_{j=1}^k \subset \mathbb{R}^d} \sum_{i=1}^n \min_{j\in\{1,\ldots, k\}} \|x_i-y_j\|^2.\label{eq:problemformulation}
\end{align}
A commonly used method is the so called k-means algorithm, which iteratively adjusts the locations of the centroids, $\{y_j\}_{j=1}^k$, and reassigns data points to the nearest centroid until convergence. 
This was introduced by Lloyd in 1957 (however, first published in \cite{lloyd1982least}) and guarantees that the objective function in \eqref{eq:problemformulation} is non-increasing in each step. However, as pointed out in the original article, the method is not aiming or shouldn't be considered to be a near optimal solution to the problem, but to give the user suggestions about how their data is organized. 

Recent development of local methods involves smart seeding of the centroid positions. One widely used such method is the k-means++ algorithm \cite{arthur2007k}, which has a $\mathcal{O}(\log k)$ multiplicative error bound to the optimal solution. However, there is also a theoretical bound \cite{kanungo2002local} of how good such local search methods may become, which motivates research on global methods. 

\subsection{Contributions}

In this paper we propose an algorithm that globally solves the k-means problem and which is computationally feasible for small $k$ (\# clusters) and $d$ (data dimension). 
The main contributions in this paper can be summarized as follows:
 \begin{itemize}

\setlength{\itemsep}{2pt}
\setlength{\parskip}{2pt}
\item We formulate the $k$-means problem as a concave assignment problem with dimension depending on $k$ and $d$.
\item We propose a new method to solve the k-means problem \eqref{eq:problemformulation}, which involves solving a sequence of alternating optimization problems, one which is linear with polynomial complexity in $n$ and one which is low-dimensional if $k$ and $d$ are small. 

\item We give a proof showing that the solution converges to a global optimal solution. Furthermore, the method provides an optimality gap for which the user can terminate the problem at a given accuracy and retrieve a solution.
\item We explore  features such as symmetry breaking, various cuts and integer constraints in order to accelerate the algorithm.
\item We illustrate the results of the method numerically and compare with state of the art global optimizers.
\end{itemize}

\subsection{Related work}
Modern implementations of the k-means algorithm often incorporate optimizations and enhancements to improve efficiency and performance. Some of the algorithms used today are Lloyd's Algorithm \cite{lloyd1982least}, 
Mini-Batch K-Means \cite{sculley2010web} and K-Means++ initialization \cite{arthur2007k}. More recent work include \cite{elkan2003using,aggarwal2009adaptive,matouvsek2000approximate,lattanzi2019better,ahmadian2019better,kanungo2002local} which aim to improve the understanding, error bounds, performance and scale. The intense research of methods to solve the problem and the vast amount of publications where it is used shows that this idea of dividing experimental data is popular.

The existing algorithms cater to various data sizes, computational resources, and efficiency requirements. The choice of algorithm depends on the specific problem, data characteristics, and available resources. While the core k-means concept remains the same, these modern variations have helped address some of the limitations and challenges of the original algorithm.

The original Lloyd's algorithm and its variations converge to a local minimum. However, the quality of the obtained solution heavily depends on the initial placement of centroids. Poor initialization might lead to suboptimal solutions or slow convergence. K-means++ initialization, introduced to mitigate this initialization problem, improves convergence guarantees in practice by selecting initial centroids that are well spread out in the data space. This initialization technique provides probabilistic bounds on the quality of the solution in terms of the approximation factor compared to the optimal solution. More specifically, Kanungo \cite{kanungo2002local} showed a multiplicative $(9+\epsilon)$ error bound, also showing that local search methods utilizing a finite amount of swaps of the assignments cannot improve the approximation below $(9-\epsilon)$ bounds. Ahmadian et al. \cite{ahmadian2019better} [Theorem 3.4] showed a multiplicative error bound of 6.3574 to the globally optimal solution. On the same topic, Lattanzi \cite{lattanzi2019better} [Theorem 1] showed a local search method that, when the gap to optimality is large enough, there is a probability to find a better solution, resulting in a solution in par with the globally optimal in $\mathcal{O}(dnk^2\log \log k)$ iterations. Lee et al. \cite{lee2017improved} showed that it is NP-hard to solve the problem below a multiplicative error of $1.0013$.

To determine a proper $k$ for the data at hand, there are several methods to compare the result when varying the parameter, such as the elbow method \cite{thorndike1953belongs} or the Silhouette method \cite{rousseeuw1987silhouettes}. These methods can help the analyst to determine which clustering is the most appropriate with their respective objective function. However, since the above mentioned algorithms are not guaranteed to find the global optimum for each $k$, it becomes a challenge to compare the result between different parameter settings.
\subsection{Introduction to the method and structure of the paper}

The main idea that this paper builds on is by rewriting the problem as a low dimensional concave assignment problem, described in Section~\ref{sec:variableclustersize}. Since a concave objective function attains its minimum at an extreme point of the bounding feasible set, we may relax the set of binary assignment matrices to real valued matrices with the same marginal constraints, without changing the optimal value or solution.

Even though the problem can be formulated in a low dimensional space, the feasible set in this image space becomes increasingly complex with increasing $n$. The idea is then to approximate this set with linear constraints where necessary, for which the problem is solvable in tractable time, see illustration in Figure~\ref{fig:outerexample}.

\begin{figure}[ht]
\vskip 0.2in
\begin{center}
\includegraphics[width=0.5\columnwidth]{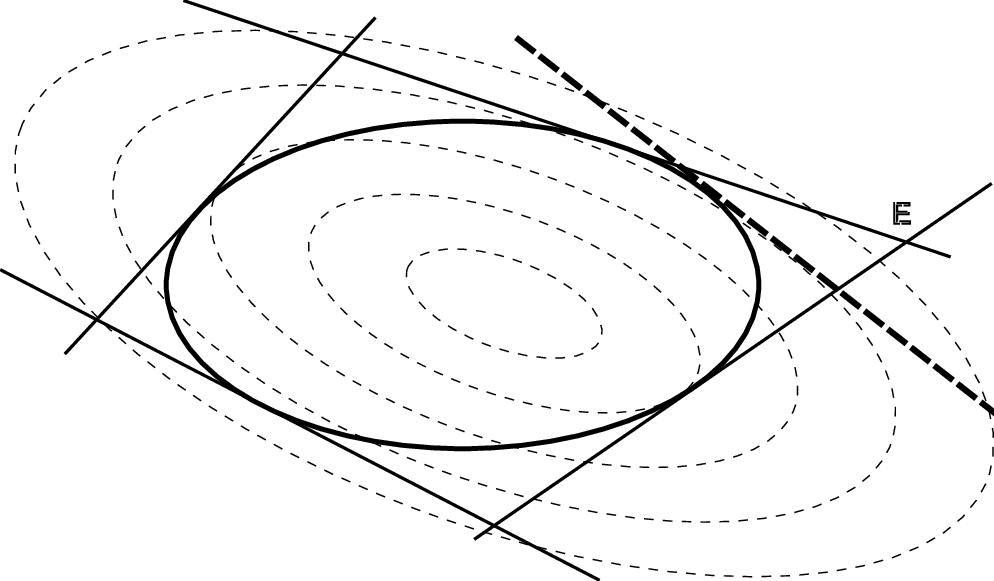}
\caption{A convex set (ellipse drawn with a solid line) is approximated by with linear cutting planes tangential to the convex set. A concave objective function (with level curves described by thin dashed ellipses)  attains its minimum in an extreme point $E$ of the outer linear approximation which is lower than any value in the ellipse. A new constraint (thick dashed line), refines the outer approximation  by removing a part which contains $E$.}
\label{fig:outerexample}
\end{center}
\end{figure}

In the procedure, a lower bound of the problem is found by solving the relaxed  problem. The optimal point in the relaxed problem is then used to find an assignment matrix using linear programming. The least upper bound of the objective in these visited assignment matrices is then an upper bound of the problem. In Section \ref{sec:convergenceproof}, we show in the setting of general smooth concave objective functions that the gap between the upper and lower bounds converge to zero, hence one globally optimal solution is found. 

In Section~\ref{sec:AcceleratingMethods} we introduce methods that accelerate the convergence of the method, and finally, in Section~\ref{sec:examples} we show with examples how the method performs in terms of accuracy and computational complexity.

\section{Problem formulation and proposed algorithm}
\label{sec:variableclustersize}

\subsection{Formulation as a low dimensional concave  problem}

In order to rewrite \eqref{eq:problemformulation} as a concave assignment problem, we first note that it can be written as 
\begin{align}
\min_{\{y_j\}_{j=1}^k\subset \mathbb{R}^d, \Gamma \in \mathcal{A} }\quad \sum_{i=1}^n\sum_{j=1}^k \|x_i-y_j\|^2 \Gamma_{i,j}, \label{eq:problemformulation2}
\end{align}
where we let $\mathcal{A}=\{\Gamma \in \{0,1\}^{n\times k}, \Gamma\ett = \ett, \Gamma^T\ett \ge \ett\}$ be a set of assignment matrices that assigns each point to a centroid.
The objective function
\begin{align}
&\sum_{i=1}^n\sum_{j=1}^k \|x_i-y_j\|^2 \Gamma_{i,j}\nonumber\\
&=\sum_{i=1}^n \|x_i\|^2+\sum_{j=1}^k \left(\sum _{i=1}^n\|y_j\|^2\Gamma_{i,j}-\sum_{i=1}^n 2\langle y_j, x_i\rangle \Gamma_{i,j}\right)\nonumber\\
&=c_0+\sum_{j=1}^k \left(\|y_j\|^2 \ett^T\Gamma e_j-2\langle y_j, X \Gamma e_j\rangle \right)\nonumber\\
&{=c_0+\sum_{j=1}^k \Big(\ett^T\Gamma e_j\|y_j - X \Gamma e_j/(\ett^T\Gamma e_j) \|^2}\nonumber\\&\qquad \qquad\qquad {-\|X\Gamma e_j\|^2/(\ett^T\Gamma e_j)\Big)}\label{eq:objective}
\end{align}
where $c_0=\sum_{i=1}^n \|x_i\|^2$ and 
$X=\begin{pmatrix} x_1,\ldots, x_n
\end{pmatrix}\in \mR^{d\times n}$, and thus the minimizing ${y_j}'s$ can be expressed explicitly as 
$y_j=X\Gamma e_j/(\ett^T\Gamma e_j)$. By substituting this into \eqref{eq:objective}, we get that \eqref{eq:problemformulation2} can be reformulated as  
\begin{align}
\min_{\Gamma \in \mathcal{A}} \quad  c_0-\sum_{j=1}^k\|X\Gamma e_j\|^2/(\ett^T\Gamma e_j).
\label{eq:relaxedformulation1}
\end{align}
Next, note that each term in the sum in \eqref{eq:relaxedformulation1} is a composition of a linear function and a norm on the form $h(\alpha, \beta):=\|\alpha\|^2/\beta$. Since $h(\alpha, \beta)$ is jointly convex in $\alpha\in \mR^n$ and $\beta\in \mR_+$ (\cite[p.89]{boyd2004convex}), it follows that the objective of \eqref{eq:relaxedformulation1} is concave function.     

Now, consider the relaxed problem, where the feasible set $\mathcal{A}$ is extended to its convex hull
\begin{align*}
\overline{\mathcal{A}}:=\{\Gamma \in [0,1]^{n\times k}, \Gamma\ett = \ett, \Gamma^T\ett \ge \ett\}.
\end{align*} 
Since the objective is concave, this does not change the minimum value. We have thus shown the following proposition. 
\begin{proposition}
The k-means problem \eqref{eq:problemformulation2} is equivalent to the concave problem
\begin{align}
\min_{\Gamma \in \overline{\mathcal{A}}}  \quad c_0-\sum_{j=1}^k\|X\Gamma e_j\|^2/(\ett^T\Gamma e_j)
\label{eq:relaxedformulation}
\end{align}
in the sense that there is a $\Gamma\in \mathcal{A}$ that is minimizes both problems, and their corresponding minimum values coincide. 
\end{proposition}

\begin{remark}
The space of assignments can be generalized to a case when each cluster has a least $n_0$ points. We would then consider the set of assignment matrices  $\mathcal{A}_{n_0}=\{\Gamma \in \{0,1\}^{n\times k}, \Gamma\ett = \ett, \Gamma^T\ett \ge n_0\ett \}$. It is straightforward to generalize our method to this case.
\end{remark}

\subsection{Relaxation and iterative scheme}

Note that by moving all the dependence on $\Gamma$ to the constraints the problem \eqref{eq:relaxedformulation} can be (trivially) reformulated as 
\begin{subequations}
\label{eq:relaxed0}
    \begin{align}
\min_{\substack{\{\hat z_j\}_{j=1}^k \subset \mR^{d}\\ \{n_j\}_{j=1}^k \subset \mR, \Gamma}}
&c_0-\sum_{j=1}^k\|\hat z_j\|^2/n_j\label{eq:relaxed0a}\\
\mbox{subject to } \; & 
\begin{pmatrix}  \hat z_j  \\n_j\end{pmatrix}=
\begin{pmatrix}
    X \Gamma e_j\\
    \ett^T \Gamma e_j
\end{pmatrix}, \; j=1,\ldots, k\label{eq:relaxed0b}\\
& \Gamma \in \overline{\mathcal{A}}.\label{eq:relaxed0c}
\end{align}
\end{subequations}
 For convenience of notation we let 
\[
Z=\begin{pmatrix}
    z_1&\ldots & z_k
\end{pmatrix},\quad  z_j=\begin{pmatrix}  \hat z_j  \\n_j\end{pmatrix} \quad \mbox{ and } \quad  \mathcal{X} := \begin{pmatrix}
X\\
\ett^T
\end{pmatrix}
\]
in which case \eqref{eq:relaxed0} becomes
\begin{subequations}
\label{eq:relaxed1}
    \begin{align}
\min_{{Z \in \mR^{(d+1)\times k}}, \Gamma}  \; &c_0-\sum_{j=1}^k\|C z_j\|^2/(Dz_j)\label{eq:relaxed1a}\\
\mbox{subject to } \; & {Z  =\mathcal{X}\Gamma}\\
& \Gamma \in \overline{\mathcal{A}}\label{eq:relaxed1c}
\end{align}
\end{subequations}
where $C:= \begin{pmatrix}
I_{d\times d} & 0_{d\times 1}
\end{pmatrix}$ and $D := \begin{pmatrix}
0_{1\times d} &1
\end{pmatrix}$. 

The set $\mathcal{F}:=\{\mathcal{X}\Gamma\mid \Gamma \in \overline{\mathcal{A}}\}$ may contain exponentially many (in $n$) extreme points and facets seen by counting the basic feasible solutions of $\overline{\mathcal{A}}$. Therefore, we consider the relaxation of \eqref{eq:relaxed1}, in terms of the constraints as  
\begin{subequations}
\label{eqn:relaxedvar}
\begin{align}
\min_{{Z \in \mathbb{R}^{(d+1)\times k}}} \quad &c_0 - \sum_{j=1}^k \|Cz_j\|^2/(Dz_j)\\
\mbox{subject to } \qquad &\langle A_r, Z\rangle \le b_r, r = 1,\ldots, N.
\label{eqn:relaxedpoly}
\end{align}
\end{subequations}

Even though \eqref{eqn:relaxedvar} is a nonconvex problem, it is low-dimensional and can be solved when $k$ and $d$ are small. The main approach then builds on iteratively approximating and refining the feasible  set $\mathcal{F}$ using constraints on the form \eqref{eqn:relaxedpoly}. To this end, we first select constraints so that \eqref{eqn:relaxedpoly} is an outer approximation of $\mathcal{F}$. Then, we solve \eqref{eqn:relaxedvar} which gives a lower bound of the problem. To refine the lower bound, we find an assignment matrix which gives an upper bound of the problem and that defines a cutting plane which refines the approximation of the feasible set. The detailed steps are described next, and the procedure is described in pseudo code in Algorithm \ref{alg:cap_new}.

To initially construct an outer approximation of $\mathcal{F}$ we consider a simplex described by the constraints
\begin{align}
  & \min_{\Gamma \in \overline{\mathcal{A}}} \left(\mathcal{X}\Gamma\right)_{ij} \le Z_{ij} \mbox{ for } i=1,..., d+1 \mbox{ and } j=1,..., k,\nonumber\\
   &  \langle \ett_{(d+1)\times k}, Z\rangle \le \max_{\Gamma \in \overline{\mathcal{A}}} \langle \ett_{(d+1)\times k}, \mathcal{X}\Gamma   \rangle.     \label{eq:boundingsimplex_2}
\end{align}
Geometrically, this simplex is described by lower bounds in every element of $Z$ and a ``diagonal'' hyperplane describing the upper bound. 

To solve \eqref{eqn:relaxedvar}, we use a half plane/vertex representation of the polytope \eqref{eqn:relaxedpoly} which effectively can be refined with new linear constraints. This method was used in \cite{ryner2023globally} to represent polytopes in order to solve a low rank concave quadratic assignment problem.

Before we continue, we note that the terms in the sum of the objective of \eqref{eqn:relaxedvar}
can be formulated using
\begin{align*}
f(z_j):= - \|Cz_j\|^2/(Dz_j).    
\end{align*}
Assuming that the optimal point to \eqref{eqn:relaxedvar} is $Z_{N}=\begin{pmatrix}
    z_1^{N} & \cdots & z_k^{N}
\end{pmatrix}$, 
we use the local gradient of the objective in this point,
\begin{align}
\mathcal{D}_{N}:= \begin{pmatrix}
\nabla f(z_1^{N}) & \cdots  & \nabla f(z_k^{N})
\end{pmatrix}
\label{eqn:localgradient}
\end{align}
where $\nabla f(z) = -\left( 2C^TC z/(Dz)- \frac{\|Cz\|^2}{(Dz)^2}D^T\right)$ to define a cutting plane. In particular, we let the normal vector of the cutting plane, $A_{N+1} := \mathcal{D}_{N}$, be the gradient and let 
\begin{align}
\label{eq:get_b}
 b_{N+1}:= \min_{\Gamma \in \overline{\mathcal{A}}}  \langle  \mathcal{D}_{N},\mathcal{X}\Gamma\rangle,   
\end{align}
and we let $\Gamma_{N}\in \mathcal{A}$ be an optimal solution to \eqref{eq:get_b}. 
By this construction, the cutting plane separates $Z_{N}$ from $\mathcal{F}$ and intersects $\mathcal{F}$, in particular 
\begin{align*}
\langle \mathcal{D}_{N} ,\mathcal X \Gamma -\mathcal{X}\Gamma_{N} \rangle \ge 0 \mbox{ for all } \Gamma \in \overline{\mathcal{A}}.
\end{align*} 
To summarize, the procedure has the following properties:
\begin{itemize}
\setlength{\itemsep}{2pt}
\setlength{\parskip}{2pt}

\item It finds an assignment matrix $\Gamma_{N}$ which gives an upper bound for the k-means problem.
\item If $Z_{N}\notin \mathcal{F}$, then the cutting plane makes $Z_{N}$ infeasible, and the constraint is added to \eqref{eqn:relaxedpoly}. 
\item If $Z_{N}\in \mathcal{F}$, then the point is optimal and $\Gamma_{N}$ is optimal to the k-means problem \eqref{eq:relaxedformulation}. 
\end{itemize}

\newcommand{\LB}{L_{\rm bound}}
\newcommand{\UB}{U_{\rm bound}}
\begin{algorithm}[H]
%\begin{algorithm}[H]
\caption{Global k-means algorithm}\label{alg:cap_new}
\begin{algorithmic}
\State \textbf{Input}  $X\in \mR^{d\times n}$, $\epsilon>0$ \hfill (Define point cloud and give tolerance level) 
\State $\LB \gets -\infty$, and $\UB \gets \infty$ \hfill(Set lower and upper bounds)
\State $(A_r, b_r)$ for $r=1,\ldots, 1+(d+1)k$ from \eqref{eq:boundingsimplex_2} \hfill (Set initial constraints)
\Repeat \\
\State $Z_N$ $\gets$ Optimal solution to \eqref{eqn:relaxedvar} \hfill (Solve \eqref{eqn:relaxedvar})
\State $\LB\gets c_0 -\sum_{j=1}^k \|Cz_{j,N}\|^2/(Dz_{j,N})$ \hfill (Update lower bound)
\State $\Gamma_N$ $\gets$ Optimal solution to \eqref{eq:get_b} \hfill (Solve \eqref{eq:get_b}) 
\State $\hat{Z} \gets \mathcal{X}\Gamma_N$
\State $\UB\!\!\gets\!\!\min(\UB,\!c_0\!-\sum_{j=1}^k\! \|C\hat{z}_{j,N}\|^2/(D\hat{z}_{j,N}))$ \hfill (Update upper bound)
\State $(A_{N+1},b_{N+1}) \gets $ Optimal solution to \eqref{eq:get_b} \hfill (Calculate new constraints from  \ref{eq:get_b})
\State $N\gets N+1$ \hfill (Update iteration number)
\Until{$\UB - \LB < \epsilon$}
\end{algorithmic}
\end{algorithm}

\section{Convergence}
\label{sec:convergenceproof}

The main result considering convergence is presented in the following theorem.
\begin{theorem} \label{thm:convergence} 
The gap between the upper bound and lower bound  in the algorithm converges to $0$ (if the tolerance is $\epsilon=0$).
\end{theorem}
\begin{proof}
Consider the $N$th iteration. Let $Z_N$ be an optimal solution to the relaxed problem \eqref{eqn:relaxedvar}, and $\Gamma_N$ is an optimal solution to the linearized problem \eqref{eq:get_b}
with corresponding point $\hat Z=\mathcal{X}\Gamma_N$. 
Note that the new constraint is given by
\begin{align}
\label{eq:constrainteq}
\langle \nabla F_{Z_N},Z-\hat{Z}\rangle \ge 0,
\end{align}
where $F(Z) =c_0-\sum_j \|Cz_j\|/D^T z_j$
is the objective function. Further, due to concavity of $F$ we have
\begin{align}
\label{eq:taylorConvex}
F(\hat{Z}) \le F(Z_N) + \langle \nabla F_{Z_N} ,\hat{Z}-Z_N\rangle.
\end{align}

Next, assume that the gap in the objective function between $\hat{Z}$ and $Z_N$ is
\begin{align}\label{eq:gap}
\epsilon_N = F(\hat{Z})-F(Z_N).
\end{align}
Then, by plugging in \eqref{eq:taylorConvex} and \eqref{eq:constrainteq} into \eqref{eq:gap} we obtain
\begin{align*}
\epsilon_N &= F(\hat{Z})-F(Z_N)\\
&\le \langle \nabla F_{Z_N},\hat{Z}-Z_N\rangle\\
&= \langle \nabla F_{Z_N},\hat{Z}-Z + Z - Z_N\rangle\\
&= \langle \nabla F_{Z_N},Z-Z_N\rangle + \langle \nabla F_{Z_N},\hat{Z}-Z\rangle\\
&\le \langle \nabla F_{Z_N},Z-Z_N\rangle \\
&\le \|\nabla F_{Z_N}\|_2\|Z-Z_N\|_2
\end{align*}
where we in the last step have used the Cauchy-Schwarz inequality.  

For any iteration number $N+k$ with $\ell>0$, we have that $Z_{N+\ell}$ is feasible for $\langle A_{N+1}, Z \rangle \le b_{N+1}$, and thus the Euclidean distance between $Z_N$ and $Z_{N+\ell}$ is at least $\epsilon_N / \|\nabla F_{Z_N}\|_2$.
If the gap in the algorithm does not converge to $0$, then there is an $\epsilon>0$ for which $\epsilon_N\ge \epsilon$ for all $N$ and thus the distance between any two points in the sequence $\{z_{N}\}_N$ is bounded from below by 
$\epsilon/(\max_{Z \in B}\ \|\nabla F_{z}\|)$, where $B$ is the initial bounding simplex for the relaxed problem.
However, since the infinite sequence of points $\{Z_{N}\}_N$ belong to a bounded set defined by $B$, there must be a convergent subsequence, which contradicts that 
there is a positive lower bound on the distance between any two points. 

Since assuming that there is a lower bound $\epsilon>0$ leads to a contradiction, it must hold that the gap between the upper and lower bound converges to zero, and thus the solution converges to a global optimum.
\end{proof}

\section{Accelerating methods}\label{sec:AcceleratingMethods}
In this section we introduce methods that accelerates Algorithm \ref{alg:cap_new}. Some methods are well known in the field of global optimization, and some of the methods utilize the structure of the k-means problem. All methods involve additional cutting planes and fits with the algorithm architecture.

\subsection{One fixed marginal constraint}
Since $\Gamma \ett=\ett$ and $Z=\mathcal{X}\Gamma$ we obtain that $Z1 = \mathcal{X}1$. This reduces the problem with $d+1$ dimensions and can be implemented as $\langle A_j,Z\rangle = b_j$ where the matrix elements $A_j-e_j\ett^T$ and $b_j = (\mathcal{X}1)_j$.

\subsection{Symmetry breaking}\label{sec:symmetrybreak}
Symmetry breaking is a known an powerful tool in global optimization, see for instance \cite{walsh2006general,hojny2019polytopes}. One obvious symmetry is that the ordering of clusters may be arbitrary. To avoid this, one may consider a fixed but arbitrary linear ordering of $z_j$, namely $a^Tz_j\le a^Tz_{j+1}$, which creates $k-1$ constraints. For simplicity, the ordering chosen is $a = \ett$.
\subsection{Branching}

A spatial hyperplane branching can be defined by $B\in \mathbb{R}^{(d+1)\times k}$ and $0 \le \beta \ll 1$, such that, $\langle B,Z\rangle \le \beta$ forms one child node and $\langle B,Z\rangle \ge -\beta$ the other. We can utilize spatial branching to reduce memory consumption when solving the low dimensional concave optimization problem \eqref{eqn:relaxedvar}, by dividing it into two problems. As our method for solving \eqref{eqn:relaxedvar} uses a vertex representation, branching reduces the number of vertexes we need to store in memory during the search. The role of $\beta$ is to have a small overlap of the sets to prevent numerical instability.

Division of the branches was done by diving the set of extreme points in two parts using the constraint $\langle v,Z \rangle \le \langle b,v\rangle$ where $b$ is the mean value of the vertices and $v$ was constructed by analyzing the second moments of the distribution of extreme points, using singular value decomposition, SVD. 

\subsection{Constrained centroid positioning}\label{sec:Convhull}
Consider the convex hull of $\{x_i\}_{i=1}^n$, constituted by the subset of points $\{\hat x_k\}_k$. Since the centroids $Y = X\Gamma \oslash 1^T\Gamma$ is a convex combination of the data points, they must be included in the set described by the convex hull, and we could restrict the centroids to this set. Handling the convex hull is computationally expensive as $n$ and $d$ increases. Let $S_i = \max_j (x_j)_i$ denote the largest value in coordinate $i$ and $s_i = \min_j (x_j)_i$ the smallest value in coordinate $i$. A coarse approximation with the bounding box is a suitable choice,
\begin{align*}
Dz_j s_i \le (z_j)_i  \le Dz_j S_i, \mbox{ for all } i = 1,\ldots,d.
\end{align*}

\subsection{Integer constraints}
The elements $(z_j)_{d+1}\in \{1,\ldots,n-(k-1)\}$ for $j = 1,\ldots,k$, since these elements contain the number of data points in each cluster. Therefore, in every iteration additional pruning of the set. That is, if the extreme points in a branch are $\{\bar{Z}_i\}_i$, we can construct cuts $\lceil \min_i ((\bar{z}_j)_{d+1})_i \rceil \le (z_j)_{d+1} \le \lfloor \max_i ((\bar{z}_j)_{d+1})_i \rfloor$, so that the bounds are integer. In combination with branching, this gave a stronger effect in our experiments.

\subsection{Tight constraints}
By exploring some fundamental properties of LP, we can derive an additional set of valid constraints. A standard LP problem can be written on the form 
\begin{align*}
\min_x\qquad  &c^Tx\\
\mbox {subject to } \; &Ax = b, x\ge 0.
\end{align*}
Suppose $x^*$ is a basic feasible solution to the LP, with basis $B$ (the complement is $B'$). Then by the optimality conditions 
the dual $y = (A_{B}^T)^{-1} c_B$ and $$s_{B'} = c_{B'}-A_{B'}^Ty > 0.$$ Any perturbation in $c_{B'}$ that makes $(s_{B'})_k < 0$ for any element index $k$ is a perturbation that changes the optimal assignment.

The problem \eqref{eq:get_b} can be stated on standard form according to
\begin{align*}
&x = \begin{pmatrix} \ve(\Gamma)\\ x_s\end{pmatrix}, \\
&c = -\begin{pmatrix}\ve(\mathcal{X}^T A_N)\\ 0_k\end{pmatrix},\\
&A = \begin{pmatrix}
\ett_k^T\otimes I_{n\times n} &0_{n\times k}\\
I_{k\times k} \otimes \ett_n^T & -I_{k\times k}
\end{pmatrix}, \\
&b = \begin{pmatrix}
\ett_n\\
\ett_k
\end{pmatrix}
\end{align*}
where $x_s$ are slack variables for the inequality constraint $x_s^T = \ett_n^T \Gamma - \ett^T_k \ge 0$.

A change in the objective in the low dimensional space will change $\partial c = \ve(\mathcal{X}^T \partial A_N)$. Thus, there is a closed form expression of the dual, namely 
\begin{align*}
s_{B'} = V \ve(A_N)
\end{align*}
where $V:=(J_{B'}-A_{B'}^T(A_B^T)^{-1})J_B) \begin{pmatrix}
I_k\otimes \mathcal{X}^T\\ 0_{(m+1)k,k}
\end{pmatrix}$ and $J_{B}$ and $J_{B'}$ are selection matrices for the set $B$ and $B'$ respectively, meaning that $(J_{B})_{i,j}= 1$ if and only if $B_i = j$.
Lets take a subset of the rows in $V$ denoted $B_2$, it is possible to find a vector $t$ such that $V_{B_2} \ve(t) = 0$ and $V_{B'_2}\ve(t) >0$. If this is possible, then $A_N+t$ is an objective for the LP, for which the same $\Gamma_N$ is optimal, but not unique, since we may take an index $i \in B_2$ and add $B'_i$ to $B$ and remove some other base. Which one is directly determined by $\Gamma \ett_k = \ett_n$, i.e., which point that is moved from one cluster to another. Thus it is possible to create a set of new constraints $\partial A_k$ by the optimization problem
\begin{align*}
\min_{s_{B'},\partial A_k} \;& \begin{pmatrix}
\ett^T& 0^T
\end{pmatrix} \begin{pmatrix}
s_{B'}\\
\ve(\partial A_k)
\end{pmatrix}\\
\mbox{subject to }\; &s_{B'} = V\ve(\partial A_k)\\
&\partial A_k^T \partial A_j \le \alpha \mbox{ for } j = 1,\ldots,k-1\\
&s_{B'} \ge 0.
\end{align*}
The tight constrains are then defined by $A_N+\partial A_j$. the role of $\alpha$ is to find new constraints that differs from each other.
\subsection{Local optimum}

Necessary optimality conditions requires that  
\begin{align*}
\Gamma^* = \argmin_{\Gamma \in \overline{\mathcal{A}}} \langle \mathcal{X}^T\mathcal{D^*},\Gamma\rangle
\end{align*}
where $\mathcal{D}^*$ is the gradient of the objective function \ref{eqn:localgradient} in the points $Z^* = \mathcal{X}\Gamma^*$. In each iteration a local search by iterating over $j$

\begin{align*}
\Gamma_{j+1}^* := \argmin_{\Gamma \in \overline{\mathcal{A}}} \langle \mathcal{X}^T\mathcal{D}_j^*,\Gamma\rangle
\end{align*}
until convergence can be applied. A fix point with this sequence is not the same as for the alternating sequence in Lloyds algorithm. However, there is no need for a distance evaluation between all centroids and points. A tight constraint surrounding the local optimum can be added to the approximate cover. 

\subsection{Least squares approximation}

Instead of the half plane cut generated by the LP problem \eqref{eq:get_b}, one may consider a least square fitting by finding the point on the projected assignment matrix closest to the point in the outer approximation by 
\begin{align*}
\min_{z} \qquad &\|z-z_N\|_F^2\\
\mbox{subject to}\;\; &z = \mathcal{X}\Gamma\\
&\Gamma \in \overline{\mathcal{A}}.
\end{align*}
We denote $\hat{z}$ the optimal argument of the above optimization problem. By iteratively adding the half plane described by 
\begin{align*}
\langle z_N-\hat{z},z-\hat{z}\rangle \le 0
\end{align*}
an approximation of the projected assignment matrix is found. Convergence of an outer half plane approximation of a convex set is well known, by iteratively applying the separating hyperplane theorem. Briefly, when $\|\hat{z}-z_N\|=0$ the optimal assignment is found. If $\|\hat{z}-z_N\|\ge \epsilon$, there is a half plane separating the point $z_N$ to $\hat{z}$, namely the plane described above because of the Hahn-Banach theorem. Suppose the distance $\|\hat{z}-z_N\| \ge \epsilon$ for all $N\ge K$, then each point in the infinite sequence has a distance between them. Hence, the sequence does not contain a convergent subsequence, which contradicts that the outer approximation is bounded.

\section{Numerical Results}
\label{sec:examples}

 Here, we present numerical results on the method. The method was implemented in Matlab, on a standard PC where the polytope connectivity was calculated on a Nvidia GTX1660 using OpenCL. The LP problem \eqref{eq:get_b} was solved using Gurobi 10 \cite{gurobi}.
 \subsection{Performance study}

We study a model problem with simulated data by considering samples taken from three normal distributions $\mathcal{N}(0,\sigma)$, $\mathcal{N}((0,2),\sigma)$ and $\mathcal{N}((2,0),\sigma)$, see Figure \ref{fig:illustrative example}. With $\sigma = 1$, this problem has no clear separability of the clusters, making the problem far more challenging.

\begin{figure}[t]
\vskip 0.2in
\begin{center}
\includegraphics[width=1.0\columnwidth]{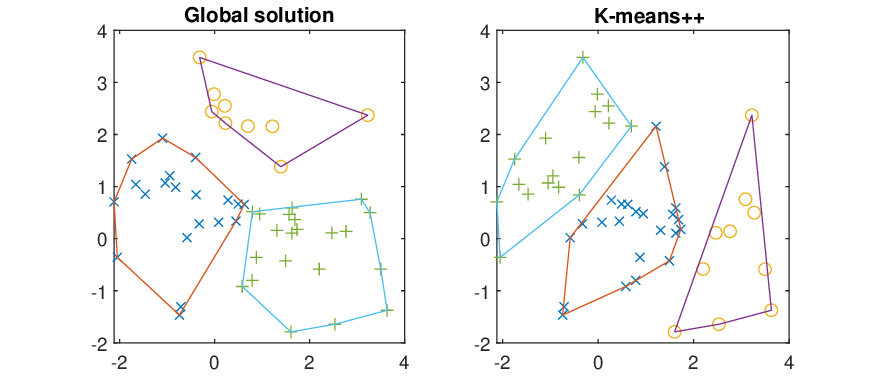}
\caption{A model problem with simulated data. Left: Exact solution to the k-means problem using the proposed algorithm. Right: Approximate solution of the problem using k-means++.}
\label{fig:illustrative example}
\end{center}
\end{figure}

\subsubsection{Performance of accelerating methods}

In this section we describe the effect on the accelerating methods when 50 points were taken from the model problem. Results are shown in Table \ref{table:compeff}. The effect of utilizing the fixed marginal which reduces the problem by $m+1$ dimensions is implemented in all examples. All experiments were run until the relative error on the gap $g$ was $10^{-4}$.
The example show that the accelerating methods makes the performance up to 3 times better.  Also, decisions of when during the sequence a strategy is put into practice affects the overall performance. No branching was performed on this particular problem, as it comes to use for large scale problems where memory consumption becomes a bottle neck.

Some observations on this particular problem include
\begin{itemize}
\setlength{\itemsep}{2pt}
\setlength{\parskip}{2pt}
\item Tight cuts drive computational cost, and should be used carefully at the very end surrounding local optima.
\item The least squares method complements the LP method for coarse cuts, but is redundant in this case to the constrained centroid cuts.
\end{itemize}

As an alternative, a restatement of the problem into a mixed integer nonlinear problem (MINLP) is as follows:
\begin{align*}
    \min_{\Gamma,\alpha,y} \;&\langle \alpha,\Gamma\rangle \\
    \mbox{subject to}\; & \|x_i-y_j\|^2 \le \alpha_{i,j}\\
    & \Gamma \in \{0,1\}^{n\times k}\\
    & \alpha \in [0,\max_{i,j} \|x_i-x_j\|^2]^{n\times k}\\
    & (y_i)_j \in [\min_l (x_l)_j,\max_l (x_l)_j]\\
    & \;\mbox{ for }j = 1,\ldots,d, \mbox{ and } \; i=1,\ldots,k.
\end{align*}
This above problem was implemented in the commercial solver Gurobi 10. However, it requires more than $10000$ seconds to solve the problem to a $25\%$ gap, which highlights the problem complexity for solving this with standard solvers. The results are well aligned with other studies that clearly show the difficulty of solving k-means by MINLP \cite{papageorgiou2018pseudo,kronqvist2021between}.

\begin{table*}[ht]
  \caption{Scaling performance of the method. Presented values are the relative gap at termination (rel. gap) and the sum of all branches for the number of iterations (Iters.), number of constraints at termination (No. constr.), number of extreme points in the approximate polytope at termination (Extr. points) and computational time (Comp. time) in seconds on a standard PC. }
  \vskip 0.15in
  \label{table:compeff2}
  \centering
  \begin{tabular}{lllll|ccccc|c}
    \toprule
    &&&& &\multicolumn{5}{c}{Proposed method}&Gurobi \\
  n&k&d&rel. gap&$\sigma$& Iters. & No. constr.&Extr. points&Branches&Comp. time&Comp. time\\
  \midrule
  50&3&2 &$10^{-4}$&1& 267 & 366&121600&1&54s&$>10000$s\\
  500&3&2&$10^{-4}$&1&310&407&161760&1&78s&-\\
  5000&3&2&$10^{-4}$&1&419&465&210856&1&1873s&-\\
  50&3&3&$10^{-3}$&1&1696&13931&18502128&45&6389s&-\\
  50&4&2&$10^{-3}$&0.1&251&3583&2745616&5&672s&-\\
  50&4&2&$10^{-3}$&0.5&2381&73680&52238384&150&1313s&-\\ 
    \bottomrule
  \end{tabular}
\end{table*}

\begin{table}[ht]
  \caption{Effect of adding various acceleration techniques. The strategies were also tested to be activated at various gaps. In the last test the strategies were active for different levels of relative gap, denoted $g$. Abbreviations are Symmetry breaking (SB), Least squares (LS), Constrained centroid (CC) and tight local optimum (TO). Presented values are number of iterations (Iters.), number of constraints at termination (No. constr.) and computational time (Comp. time) in seconds on a standard PC.}
  \label{table:compeff}
  \vskip 0.15in
  \centering
  \begin{tabular}{l|lll}
    \toprule
  & Iters. & No. &Comp.\\
  &  &constr. & time\\
  \midrule
  Original & 546&552&151s\\
  SB &406&414&83s\\
  SB +  LS & 189& 370&71s\\
  SB +  CC &293& 313&54s\\
  SB + CC +LS & 150 & 298&52s\\
  SB + CC +LS ($g>1$) & 267 &  366&54s\\
  +TO ($g<0.01$)&&&\\
    \bottomrule
  \end{tabular}
\end{table}

\subsubsection{Cluster separability}
The complexity of finding the optimal clustering depends on the separation between the clusters as described in Table \ref{table:sigmaeff}. The test was performed on the model problem with 500 sampled points.

\begin{table}[h]
  \caption{Effect on the separability of clusters. With increased separability the computational effort is reduced significantly. }
  \label{table:sigmaeff}
  \vskip 0.15in
  \centering
  \begin{tabular}{l|lll}
    \toprule
  $\sigma$ & Iters. & No. &Comp.\\
  &  &constr. & time\\
  \midrule
  1 (no separability)&310 &  407&78s\\
  0.5 (visible separability) &96&152&23s\\
  0.2 (clear separability)& 24&   72&9s\\
    \bottomrule
  \end{tabular}
\end{table}

\subsubsection{Scaling}
Here, we describe scaling of our method. As the QP problem has exponential complexity in $d$ and $k$ this is the bottleneck. The LP scales polynomially in $n$ having less impact on the computational time and memory consumption. The branching was performed when the number of extreme points exceeded $5\cdot 10^{5}$ in a branch to match memory constraints on the computer on which the tests were run. We used the same test problems, and the results are shown in Table \ref{table:compeff2} .

\subsection{MNIST numerals}

In this section we compare the global solution of the k-means problem on MNIST \cite{lecun1998mnist} numerals "7" and "9" compared to k-means++. As an underlying distance we use the Wasserstein distance, and the distance map is then subject to a doubly stochastic diffusion \cite{landa2021doubly} to reduce noise and to separate intertwined distributions. Representative low dimensional point clouds preserving the distances are generated using multidimensional scaling, i.e. with a distance between numerals $x_i$ and $x_j$ $d(x_i,x_j)$ we create a Euclidean point cloud representation $d(x_i,x_j)^2 = \|y_i-y_j\|_2^2$. Results in terms of the k-means objective score, Purity and NMI are presented in Table \ref{table:numerals} and the clustering is presented in Figure \ref{fig:numerals}. The example shows that when the clusters are not well separated, there are close to optimal assignments that differs significantly in their structure.

\begin{table}[ht]
  \caption{K-means clustering of 100 examples of the numerals "7" and 100 examples of "9".}
  \vskip 0.15in
  \label{table:numerals}
  \centering
  \begin{tabular}{l|lll}
    \toprule
  & k-means & Purity &NMI\\
  & objective & & \\
  \midrule
  Proposed algorithm& 278.2523&0.81&0.2985\\
  K-means++ & 281.9968 & 0.67&0.0915\\
    \bottomrule
  \end{tabular}
\end{table}

\begin{figure}[H]
\vskip 0.2in
\begin{center}
\includegraphics[width=0.32\columnwidth]{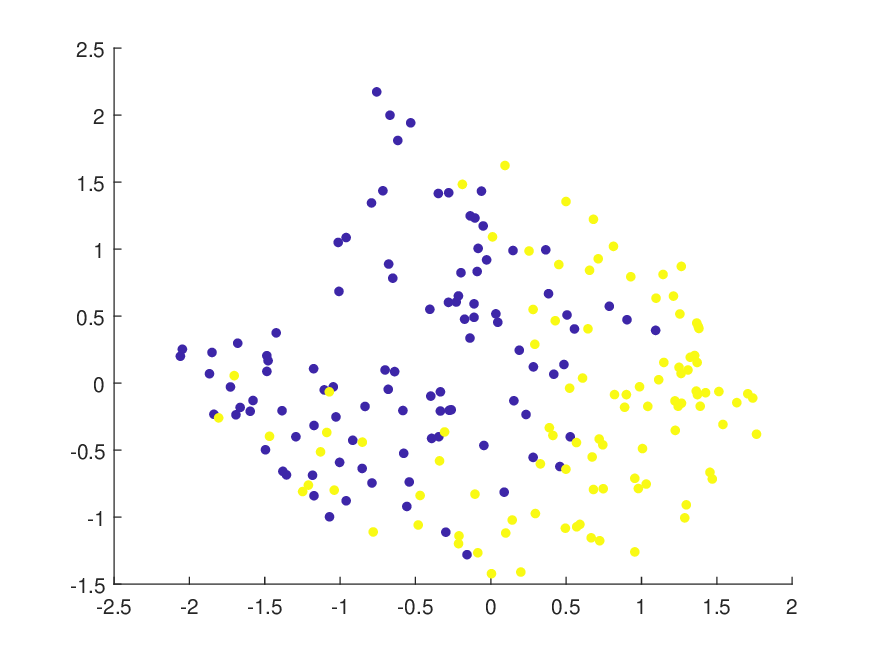}
\includegraphics[width=0.32\columnwidth]{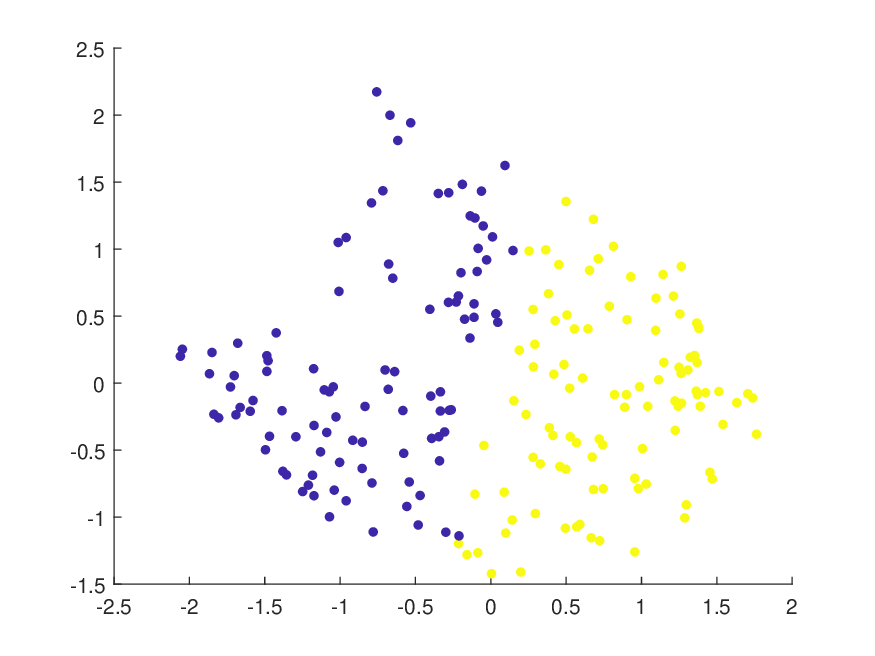}
\includegraphics[width=0.32\columnwidth]{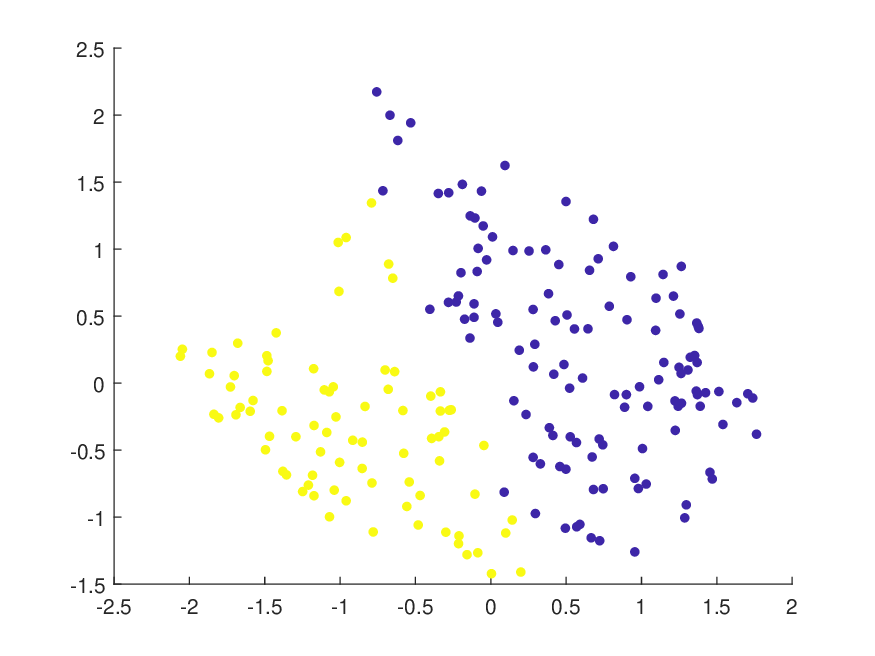}
\caption{Clustering differences using exact solution of the k-means problem and k-means++. The images show the two most significant dimensions of the multi dimensional scaling. Left image: Ground truth, middle image: Exact solution, right image: k-means++. The exact solution tends to follow the ground truth, whilst the k-means++ has found a local optimum which is not describing the clusters correctly.}
\label{fig:numerals}
\end{center}
\end{figure}

\section{Discussion}
This paper describes a method to globally solve low dimensional k-means problems. The numerical results indicate that local optimization algorithms may find local solutions far from the global optimum.
For the MNIST test problem, the global solution achieved a significantly better correspondence with the intended classification compared to the local method. Keep in mind, if the data is not well suited for k-means classification, then there are no guarantees that a global solution gives a better classification. 

In order to further strengthen the methodology, the area of polytope approximations and pruning should be investigated.
The linear optimizaton problem \eqref{eq:get_b} could be further accelerated using Sinkhorn iterations with a positive dual for the inequality constrained marginal $\Gamma^T\ett\ge \ett$, requiring a cross over methodology to transfer from an interior point estimate to a binary extreme point, e.g., by modifying the ideas in \cite{kaijser1998computing}. The branching technique was implemented in series for the numerical examples. A parallel implementation would increase the performance linearly.

%Bibliography
\bibliographystyle{unsrt}  
\bibliography{references}

\end{document}